\documentclass[12pt,a4paper, twoside, reqno]{amsart} 

\usepackage[margin=3cm]{geometry}
\usepackage{ifpdf}
\usepackage{graphicx}
\usepackage[T1]{fontenc} 
\usepackage[latin1]{inputenc} 
\usepackage{amsmath} 
\usepackage{amsthm} 
\usepackage{amsfonts} 
\usepackage{amssymb}
\usepackage{amsbsy}
\usepackage{bbm} 
\usepackage[english]{babel} 
\usepackage{varioref} 
\usepackage{enumerate} 

\theoremstyle{definition}

\newtheorem*{defn*}{Definition}

\theoremstyle{remark}

\theoremstyle{plain}

\newtheorem*{thm*}{Theorem}

\newcommand{\norm}[1]{\ensuremath{\left\Vert #1 \right\Vert}}

\newcommand{\infabs}[1]{\ensuremath{\left\vert #1 \right\vert}}

 \newtheorem{theorem}{Theorem}
 \newtheorem{lemma}{Lemma}
 
 \newtheorem{corollary}{Corollary}

\newcommand{\mc}{\mathcal}
 \newcommand{\rar}{\rightarrow}

 \newcommand{\B}{\mc{B}}
 
 \newcommand{\D}{\mc{D}}

 \newcommand{\M}{\mc{M}}
 \newcommand{\PP}{\mc{P}}
 \newcommand{\MM}{\mathfrak{M}}
 \newcommand{\EE}{\mathfrak{E}}

 \newcommand{\R}{\mathbb{R}}
 \newcommand{\N}{\mathbb{N}}
 \newcommand{\Q}{\mathbb{Q}}
 \newcommand{\Z}{\mathbb{Z}}
 \newcommand{\udim}{\overline{\mathrm{dim}}}
\newcommand{\bad}{\mathrm{Bad}}

\title[Mixed Littlewood for pseudo-absolute values]{The mixed Littlewood conjecture for pseudo-absolute values}

\author{Stephen Harrap}

\address{S. Harrap, Department of Mathematics, University of York,
  Heslington, York YO10 5DD, United Kingdom}

\author{Alan Haynes}

\address{A. Haynes, Department of Mathematics, University of Bristol, University Walk,
  Bristol BS8 1TW, United Kingdom}

\subjclass[2000]{37A45, 11K60, 11J83, 11J86}
\keywords{Littlewood conjecture, measure rigidity, linear forms in logarithms}
\thanks{AH:~Research supported by EPSRC grant EP/F027028/1, and by a fellowship from the Heilbronn Institute for Mathematical Research}

\begin{document}

\begin{abstract}
In this paper we study the Mixed Littlewood Conjecture with pseudo-absolute values. We show that if $p$ is a prime and $\D$ is a pseudo-absolute value sequence satisfying mild conditions then
\[\inf_{n\in\N} n|n|_p|n|_\D\|n\alpha\|=0\quad\text{for all}\quad\alpha\in\R.\]
Our proof relies on a measure rigidity theorem due to Lindenstrauss and lower bounds for linear forms in logarithms due to Baker and W\"{u}stholz. We also deduce the answer to the related metric question of how fast the infimum above tends to zero, for almost every $\alpha$.
\end{abstract}

\maketitle

\section{Introduction}
\label{sec:introduction}

For $\alpha\in\R$ let $\|\alpha\|$ denote the distance from $\alpha$ to the nearest integer. The Littlewood Conjecture is the assertion that for every $\alpha , \beta\in\R$,
\begin{equation}\label{eqn:Littlewood}
	\inf_{n\in\N} n \norm{n \alpha} \norm{n\beta} \: = \: 0.
\end{equation}
This conjecture has come to light recently because of its connection to measure rigidity problems for diagonal actions on the space of unimodular lattices. This connection was exploited by Einsiedler, Katok, and Lindenstrauss \cite{EKL} to show that the set of pairs $(\alpha, \beta)\in\R^2$ which do not satisfy (\ref{eqn:Littlewood}) has Hausdorff dimension zero.

More recently de Mathan and Teulié \cite{dMT} have proposed a problem which is closely related to the Littlewood Conjecture. Let $\mathcal{D}=\{n_k\}_{k\ge 0}$ be an increasing sequence of positive integers with $n_0=1$ and $n_k|n_{k+1}$ for all $k$. We refer to such a sequence as a {\em pseudo-absolute value sequence}, and we define the {\em $\mathcal{D}$-adic pseudo-absolute value} $|\cdot|_\D:\N\rar \{n_k^{-1}:k\ge 0\}$ by
 \[
	\infabs{n}_\mathcal{D} = \min\{ n_k^{-1} : n\in  n_k\mathbb{Z} \}.
\]
In the case when $\D=\{a^k\}_{k=0}^\infty$ for some integer $a\ge 2$ we also write $|\cdot|_\D=|\cdot|_a$. If $p$ is a prime then $|\cdot|_p$ is the usual $p-$adic absolute value.

The de Mathan and Teulié Conjecture, which we will refer to as the Mixed Littlewood Conjecture, is the assertion that for any pseudo-absolute value $|\cdot|_\D$ and for every $\alpha\in\R$,
\begin{equation}\label{eqn:MixedLittlewood}
	\inf_{n\in\N} n \infabs{n}_\mathcal{D}\norm{n \alpha}  \: = \: 0.
\end{equation}
The distribution of values of the quantities $|n|_\D$ mimics the distribution of values of $\|n\beta\|$, for suitably chosen $\beta$. In the case when $\D=|\cdot|_a$ for some integer $a\ge 2$ the Mixed Littlewood Conjecture also has a dynamical formulation in terms of the action of a certain diagonal group on a quotient space of
\[\mathrm{SL}_2(\R)\times\prod_i\mathrm{SL}_2(\Q_{p_i}),\]
where $\{p_i\}$ is the collection of primes dividing $a$. By employing measure rigidity results in this setting Einsiedler and Kleinbock \cite{EK} proved that when $|\cdot|_\D=|\cdot|_a$ the set of $\alpha\in\R$ which do not satisfy (\ref{eqn:MixedLittlewood}) has Hausdorff dimension zero.

The case of the Mixed Littlewood Conjecture with more than one $p-$adic or pseudo-absolute value has also been a topic of recent interest. If $\D_1$ and $\D_2$ are two pseudo-absolute value sequences then it is reasonable to conjecture that for any $\alpha\in\R$,
\begin{equation}\label{eqn:MixedLittlewood2}
\inf_{n\in\N}n|n|_{\D_1}|n|_{\D_2}\|n\alpha\|=0.
\end{equation}
It is shown in \cite{EK} that the Furstenberg Orbit Closure Theorem \cite[Theorem IV.1]{Fur} implies that (\ref{eqn:MixedLittlewood2}) is true whenever $\D_1=\{a^k\}$ and $\D_2=\{b^k\}$ for two multiplicatively independent integers $a$ and $b$. This result was strengthened by Bourgain, Lindenstrauss, Michel, and Venkatesh \cite{BLMV} who proved a result which implies (see \cite[Section 4.6]{BHV}) that there is a constant $\kappa>0$ such that for all $\alpha\in\R$,
\[\inf_{n\in\N}n(\log\log\log n)^\kappa|n|_a|n|_b\|n\alpha\|=0.\]
These results rely on understanding the dynamics of semigroups of toral endomorphisms. They provide a contrast to the situation of the original Littlewood Conjecture, where nothing seems to be gained by adding more real variables.

It was pointed out by Einsiedler and Kleinbock in \cite{EK} that the dynamical machinery used to study these problems does not readily extend to the case of more general pseudo-absolute values. Our first result in this paper demonstrates how recent measure rigidity theorems can be combined with bounds for linear forms in logarithms to obtain more general results.
\begin{theorem}\label{thm:psav}
Suppose that $a\ge 2$ is an integer and that $\D=\{n_k\}$ is a pseudo-absolute value sequence all of whose elements are divisible by finitely many fixed primes coprime to $a$. If there is a $\delta\ge 0$ with
\begin{equation}\label{eqn:psavthmhyp}
\log n_k\le k^\delta~\text{ for all }~k\ge 2,
\end{equation}
then for any $\alpha\in\R$ we have that
\begin{equation}\label{eqn:psavthmconc}
\inf_{n\in\N}n|n|_a|n|_\D\|n\alpha\|=0.
\end{equation}
\end{theorem}
Our proof of this theorem is inspired in part by Furstenberg's original proof of his Orbit Closure Theorem \cite{Fur}, and by the ideas used by Bourgain, Lindenstrauss, Michel, and Venkatesh in \cite{BLMV}.  Of particular interest is the case when consecutive elements of the sequence $\D$ have bounded ratios (cf. \cite{BV,BDdM,EK,dMat,dMT}), and we will say that $\D$ and $|\cdot|_\D$ have bounded ratios in this case. This roughly corresponds in the original Littlewood Conjecture to having
\begin{equation*}
\inf_{n\in\N}n\|n\beta\|>0,
\end{equation*}
which is indeed the only interesting case of that conjecture anyway. For the bounded ratios case our theorem gives a quite satisfactory answer to the problem at hand.
\begin{corollary}\label{cor:psav}
Suppose that $a\ge 2$ is an integer and that $\D$ is a pseudo-absolute value sequence with bounded ratios, all of whose elements are coprime to $a$. Then for any $\alpha\in\R$ we have that
\[\inf_{n\in\N}n|n|_a|n|_\D\|n\alpha\|=0.\]
\end{corollary}
After establishing Theorem \ref{thm:psav} we will turn to the problem of determining the almost everywhere behavior of the quantities on the left hand side of (\ref{eqn:MixedLittlewood}). The analogue of this problem for the Littlewood Conjecture was established by Gallagher \cite{Gal} in the 1960's. He proved that if
$\psi:\mathbb{N} \rightarrow \mathbb{R}$ is any non-negative decreasing function for which
\begin{equation}\label{eqn:galdivcond}
\sum_{n \in \mathbb{N}} \log(n)\psi(n)=\infty
\end{equation}
then for almost every $(\alpha, \beta)\in\R^2$
\begin{equation}\label{eqn:galconclusion}
	 \norm{n \alpha} \norm{n\beta} \: \le \: \psi(n)~\text{for infinitely many}~n\in\N.
\end{equation}
For example this shows that for almost every $(\alpha,\beta)\in\R^2$ we can improve (\ref{eqn:Littlewood}) to
\[	\inf_{n\in\N} n(\log n)^2(\log\log n) \norm{n \alpha} \norm{n\beta} \: = \: 0.\]
Although Gallagher's method does not readily apply to the mixed problems that we are considering, it has recently been shown using other techniques \cite{BHV} that if $p$ is a prime, if $\psi$ is as above, and if (\ref{eqn:galdivcond}) holds then for almost every $\alpha\in\R$,
\begin{equation*}
	 |n|_p\norm{n \alpha} \: \le \: \psi(n)~\text{for infinitely many}~n\in\N.
\end{equation*}
Here we will show how this result can be extended to non $p-$adic pseudo-absolute values $|\cdot|_\D$. The quality of approximation that we obtain will necessarily depend on the rate at which the sequence $\D$ grows. For this reason, given a pseudo-absolute value sequence $\D$ we define $\M:\N\rar\N\cup\{0\}$ by
\[\M(N)=\max\left\{k: n_k \leq N  \right\}.\]
\begin{theorem}\label{thm:2}
Suppose that $\psi:\mathbb{N} \rightarrow \mathbb{R}$ is non-negative and decreasing and that $\mathcal{D}=\{n_k\}$ is a pseudo-absolute value sequence satisfying
	\begin{equation}\label{eqn:thm2.Dhyp}
		\sum_{k=1}^{\mathcal{M}(N)} \frac{\varphi(n_k)}{n_k}\gg \mathcal{M}(N)~\text{for all}~N\in\N,
	\end{equation}
where $\varphi$ denotes the Euler phi function. Then for almost all $\alpha\in\R$ the inequality
	\begin{equation}
		\label{eqn:weak}
		\infabs{n}_\mathcal{D} \norm{n\alpha} \le \psi(n)
	\end{equation}
	has infinitely (resp.\ finitely) many solutions $n \in \mathbb{N}$ if the sum \begin{equation}\label{eqn:thm2divcond}
\sum_{n=1}^{\infty} \mathcal{M}(n)\psi(n)
\end{equation}
diverges (resp.\ converges).
\end{theorem}
We also note that when (\ref{eqn:thm2divcond}) converges the inequality (\ref{eqn:weak}) always has finitely many solutions, regardless of whether or not (\ref{eqn:thm2.Dhyp}) is satisfied.

When $|\cdot |_D=|\cdot |_p$ for some prime $p$ we have that $\M(N)\asymp \log N$, and Theorem \ref{thm:2} reduces in this case to the previously mentioned result from \cite{BHV}. To see what Theorem \ref{thm:2} means in terms of the infima type expressions that occur in the Mixed Littlewood Conjecture, if $\D$ satisfies (\ref{eqn:thm2.Dhyp}) then for almost every $\alpha\in\R$ we have that
\begin{equation*}
	\inf_{n \rightarrow \infty} n \M(n)(\log n)(\log\log n) \infabs{n}_\mathcal{D}\norm{n \alpha}  \: = \: 0,
\end{equation*}
while on the other hand for any $\epsilon>0$ and for almost every $\alpha\in\R,$
\begin{equation*}
	\inf_{n \rightarrow \infty} n \M(n)(\log n)(\log\log n)^{1+\epsilon} \infabs{n}_\mathcal{D}\norm{n \alpha}  \: > \: 0.
\end{equation*}
Furthermore the hypothesis on $\D$ in Theorem \ref{thm:2} is not that restrictive in practice. Although it is possible to choose $\D$ so that (\ref{eqn:thm2.Dhyp}) does not hold, any reasonably chosen pseudo-absolute value sequence should satisfy the condition. In particular if $\D$ has bounded ratios or even if the elements of $\D$ are divisible only by some finite collection of primes then it is easy to check that (\ref{eqn:thm2.Dhyp}) is satisfied. For the interested reader we will indicate in Section \ref{sec:conclusion} how one can construct a sequence $\D$ for which (\ref{eqn:thm2.Dhyp}) fails.

{\bf Acknowledgements:} We would like to thank Sanju Velani for encouraging us to look at these problems, which seem to have first been proposed in a systematic way in \cite[Section 1.3]{BV}. The second author would like to thank Barak Weiss  for helpful conversations regarding this project.


\section{Preliminaries for the proof of Theorem \ref{thm:psav}} \label{sec:Preliminaries1}

\subsection{Invariant measures for continuous transformations}

Suppose $X$ is a compact metric space and let $\B$ denote the $\sigma$-algebra of Borel subsets of $X$. Let $\mathfrak{M}=\MM(X)$ be the set of all probability measures on $(X,\B)$, and if $T:X\rar X$ is a continuous map let $\mathfrak{M}_T=\MM_T(X)$ be the subset of $T-$invariant measures in $\mathfrak{M}.$ In other words
\[\mathfrak{M}_T=\{\mu\in\MM~:~\mu(B)=\mu(T^{-1}B)\text{ for all }B\in\B\}.\]
The set $\MM$ has a natural topology coming from the Riesz Representation Theorem, and we refer to this as the weak$^*$ topology on $\MM$. The following basic lemma summarizes some of the important properties of this topology (see \cite[Theorems 6.4, 6.5, 6.10]{Wal} for proofs).
\begin{lemma}
If $X$ is a compact metric space then we have that:
\begin{enumerate}
\item[(i)] The space $\MM$ is compact and metrizable in the weak$^*$ topology,
\item[(ii)] The set $\MM_T$ is a non-empty, closed, convex subset of $\MM$, and
\item[(iii)] The extreme points of $\MM_T$ are exactly the measures $\mu\in\MM$ for which $T$ is an ergodic measure preserving transformation of $(X,\mu)$.
\end{enumerate}
\end{lemma}
Let $\EE_T=\EE_T(X)$ be the subset of extreme points of $\MM_T(X)$. Since $\MM$ is metrizable and $\MM_T$ is compact and convex, by the Choquet Representation Theorem \cite[Chapter 3]{Phe} for any $\mu\in\MM_T$ there is a probability measure $\lambda$ supported on $\EE_T$ with the property that
\begin{equation}\label{eqn:ergdecomp}
\mu=\int_{\EE_T}m~d\lambda(m).
\end{equation}
This is the {\em ergodic decomposition of $\mu\in\MM_T$}.
\subsection{Entropy and dimension}
Suppose that $X$ is a compact metric space with metric $d$ and that $T:X\rar X$ is continuous. For $n\in\N$ and $\epsilon >0$ we say that a subset $A\subseteq X$ is {\em $(n,\epsilon)$-separated with respect to $T$} if for any $\alpha,\beta\in A, \alpha\not=\beta,$ we have that
\[\max_{0\le i\le n-1}d(T^i\alpha,T^i\beta)\ge\epsilon.\]
Let $s_n(T,\epsilon)$ be the largest cardinality of an $(n,\epsilon)-$separated subset of $X$ with respect to $T$. The \emph{topological entropy of $T$} is defined as
\[h_{\text{top}}(T)=\lim_{\epsilon\rar 0}~\limsup_{n\rar\infty}\frac{\log s_n(T,\epsilon)}{n}.\]

Next if $\mu\in\MM_T$ and $\PP\subseteq\B$ is a finite partition of $X$ we set
\[H_\mu(\PP)=-\sum_{P\in\PP}\mu(P)\log\mu (P),\]
and we let
\[h_\mu(T,\PP)=\lim_{n\rar\infty}\frac{1}{n}H_\mu\left(\bigvee_{i=0}^{n-1}T^{-i}\PP\right),\]
where
\[\bigvee_{i=0}^{n-1}T^{-i}\PP=\left\{\bigcap_{i=0}^{n-1}T^{-i}P_i:P_0,\ldots ,P_{n-1}\in\PP\right\}.\]
The \emph{metric entropy of $T$ with respect to $\mu$} is defined as
\[h_\mu(T)=\sup_{\PP}h_\mu(T,\PP),\]
where the supremum is taken over all finite partitions $\PP\subseteq\B$. When there is no confusion we may also refer to $h_\mu(T)$ as the entropy of $\mu$.

The map from $\MM_T$ to $[0,\infty)$ which sends $\mu$ to $h_\mu$ is affine \cite[Theorem 8.1]{Wal}. Also the topological and metric entropies associated to a continuous transformation are connected by the formula
\begin{equation}\label{eqn:varprinc}
h_{\text{top}}(T)=\sup\{h_\mu(T):\mu\in\MM_T\},
\end{equation}
which is known as the variational principle \cite[Theorem 8.6]{Wal}.

A concept which is somewhat related to topological entropy is the notion of the upper box dimension of a subset $A\subseteq X$. For $\epsilon>0$ we say that $B\subseteq A$ is $\epsilon-$separated if for any $\alpha,\beta\in B, \alpha\not=\beta,$ we have that $d(\alpha,\beta)\ge\epsilon$. Let $s(A, \epsilon)$ be the largest cardinality of an $\epsilon-$separated subset of $A$. The \emph{upper box dimension of $A$} is defined as
\[\udim (A)=\limsup_{\epsilon\rar 0}\frac{\log s(A,\epsilon)}{\log(1/\epsilon)}.\]
First we establish an elementary fact. Here and in what follows we are working in the metric space $(\R/\Z,\|\cdot\|)$.
\begin{lemma}\label{lem:epssepdiffset}
For any $A\subseteq\R/\Z$ and $\epsilon>0$ we have that
\[s(A-A,\epsilon)\le2s(A,\epsilon)^2,\]
where $A-A=\{\alpha-\beta:\alpha,\beta\in A\}$.
\end{lemma}
\begin{proof}
Given $\epsilon>0$ let $\{\alpha_1,\ldots , \alpha_k\}$ be an $\epsilon-$separated subset of $A$ with maximum cardinality. Then we have that
\[A\subseteq \bigcup_{1\le i\le k}B(\alpha_i,\epsilon),\]
where $B(\alpha_i,\epsilon)$ denotes the open ball of radius $\epsilon$ centered at $\alpha_i$. This gives that
\[A-A\subseteq \bigcup_{1\le i,j\le k}(B(\alpha_i,\epsilon)-B(\alpha_j,\epsilon))=\bigcup_{1\le i,j\le k}B(\alpha_i-\alpha_j,2\epsilon),\]
and therefore
\[s(A-A,\epsilon)\le\sum_{1\le i,j\le k}s(B(\alpha_i-\alpha_j,2\epsilon),\epsilon)=2s(A,\epsilon)^2.\]
\end{proof}
In our proof of Theorem \ref{thm:psav} we will link upper box dimension and entropy using following lemma.
\begin{lemma}\label{lem:dimentropy}
Suppose that $a\in\N, a\ge 2$, and let $T_a:\R/\Z\rar\R/\Z$ be the map $T_a(x)=ax.$ If $A\subseteq \R/\Z$ is a closed set satisfying $T(A)\subseteq A$ then for any $\epsilon>0$ there exists a measure $\mu\in\MM_{T_a}(A)$ with
\[h_\mu (T_a)\ge \udim (A)\cdot\log a-\epsilon. \]
In particular if $\udim (A)>0$ then there is a measure $\mu\in\EE_{T_a}(A)$ with positive entropy.
\end{lemma}
\begin{proof}
Let $d=\udim (A)$ and assume without loss of generality that $d>0$. Choose $\{\epsilon_m\}\subseteq\R$, decreasing to $0$, with
\[d=\lim_{m\rar\infty}\frac{\log s(A,\epsilon_m)}{\log (1/\epsilon_m)}.\]
Then for any $0<\delta<d$ there is an integer $m_0$ with
\[s(A,\epsilon_m)\ge (1/\epsilon_m)^{d-\delta}\text{ for all }m\ge m_0.\]
Now for the moment fix a $\delta$ and an $m\ge m_0$ and let $n$ be the integer which satisfies $a^{-n}\le\epsilon_m<a^{-n+1}$. Then if $\{\alpha_1,\ldots ,\alpha_k\}$ is an $\epsilon_m$-separated subset of $A$ of maximum cardinality we have that $k\ge a^{(n-1)(d-\delta)}$ and that
\[\|\alpha_i-\alpha_j\|\ge a^{-n}\text{ for all }1\le i<j\le k.\]
It is not difficult to check that the latter condition implies that for any $i\not= j$ we can find an integer $0\le \ell<n$ with \[\|a^\ell\alpha_i-a^\ell\alpha_j\|\ge 1/2a.\]
In other words the set $\{\alpha_1,\ldots ,\alpha_k\}$ is $(n,1/2a)-$separated with respect to $T_a$. This gives that
\[\frac{\log s_n(T_a|_A,1/2a)}{n}\ge(d-\delta)\log a-\frac{(d-\delta)\log a}{n}.\]
Now our choice for $n$ must tend to infinity with $m$ and this gives that
\[\limsup_{n\rar\infty}\frac{\log s_n(T_a|_A,1/2a)}{n}\ge (d-\delta)\log a.\]
Finally by letting $\delta$ tend to zero we obtain that
\[h_{\text{top}}(T_a|_A)\ge \udim (A)\cdot \log a.\]
The first claim of the lemma then follows from the variational principle (\ref{eqn:varprinc}). For the second claim let $\mu$ be any measure in $\MM_{T_a}(A)$ with positive entropy. Using the ergodic decomposition (\ref{eqn:ergdecomp}) and the fact that entropy is affine we have that
\[h_\mu(T_a)=\int_{\EE_{T_a}}h_m(T_a)d\lambda(m).\]
Since this integral is positive there must be a collection of ergodic measures, of positive measure with respect to $\lambda$, which have positive entropy. This finishes the proof of the lemma.
\end{proof}

\subsection{Diophantine approximation} For each $c>0$ we define $\bad (c)\subseteq \R$ to be the collection of real numbers $\alpha$ which satisfy
\[\inf_{n\in\N} n\|n\alpha\|> c.\]
We say that a real number $\alpha$ is {\em badly approximable} if $\alpha\in\cup_{c>0}\bad (c)$, and we say that $\alpha$ is {\em well approximable} otherwise. The sets $\bad (c)$ are invariant under integer translation and so we also think of them, as well as the sets of badly and well approximable numbers, as subsets of $\R/\Z$.

From the classical theory of continued fractions it has long been known that almost every $\alpha$, with respect to Lebesgue measure, is well approximable \cite{Ber,Bor}. Recently Einsiedler, Fishman, and Shapira, using a measure rigidity theorem due to Lindenstrauss \cite{Lin}, have shown that we may draw the same conclusion with Lebesgue measure replaced by any times-$a$ invariant measure with positive entropy.
\begin{theorem}{\cite[Theorem 1.4]{EFS}}\label{thm:EFS}
Suppose $a\in\N$ and let $T_a:\R/\Z\rar\R/\Z$ be the map $T_a(x)=ax$. If $\mu\in\EE_{T_a}$ has positive entropy then $\mu-$almost every $\alpha\in\R/\Z$ is well approximable.
\end{theorem}

Finally we say that $a_1,\ldots ,a_s\in\N$ are \emph{multiplicatively independent} if the real numbers $\log a_1,\ldots ,\log a_s$ are linearly independent over $\Q$. We will use the following deep theorem of Baker and W\"{u}stholz on lower bounds for linear forms in logarithms.
\begin{theorem}\cite{BW}\label{thm:BW}
Suppose that $a_1,\ldots , a_s\in\N$ are multiplicatively independent. Then there exists a constant $\kappa>0$, which depends only on $a_1,\ldots ,a_s$, such that for any $b_1,\ldots ,b_s\in\Z$, not all $0$, we have that
\begin{equation*}
\left|\sum_{r=1}^sb_r\log a_r\right|\ge \left(3\cdot\max_{1\le r\le s}|b_r|\right)^{-\kappa}.
\end{equation*}
\end{theorem}


\section{Proof of Theorem \ref{thm:psav}}\label{sec:ProofOfTheorem1}
Let $\Sigma_a=\{a^\ell\}_{\ell\ge 0}$ and let $\Sigma_a\D=\{a^\ell n_k\}_{\ell,k\ge 0}$. For $\alpha\in\R$ let $A(\alpha)\subseteq\R/\Z$ denote the closure of the set $(\Sigma_a\D)\alpha=\{a^\ell n_k\alpha\}_{\ell,k\ge 0}\subseteq\R/\Z$. If $\alpha\in\Q$ then (\ref{eqn:psavthmconc}) is trivially satisfied, so for the remainder of the proof we will assume that $\alpha\not\in\Q$.

Now suppose that for some $\alpha\in\R$ there were a constant $c>0$ such that
\[\inf_{n\in\N}n|n|_a|n|_\D\|n\alpha\|>c.\]
Then for any $\ell,k\ge 0$ we would have that
\[\inf_{n\in\N}n\left\|n(a^\ell n_k\alpha)\right\|\ge\inf_{n\in\N}\left(a^\ell n_kn\right)\left|a^\ell n_kn\right|_a\left|a^\ell n_kn\right|_\D\left\|a^\ell n_kn\alpha\right\|>c.\]
In other words we would have that $(\Sigma_a\D)\alpha\subseteq\bad (c)$, which would then imply that the set $A(\alpha)$ does not contain any well approximable points. Therefore in order to prove Theorem \ref{thm:psav} we just have to show that for any $\alpha\in\R\setminus\Q$ the set $A(\alpha)$ contains a well approximable point.

First we will show that we can always find long sequences of integers in $\Sigma_a\D$ with ratios close to $1$ (see equation (\ref{eqn:t_j/t_i}) below). Let $\{p_1<\cdots <p_s\}$ be the collection of prime numbers which divide the elements of $\D$, and for each $k\ge 0$ write
\[n_k=p_1^{b^{(1)}_k}\cdots p_s^{b^{(s)}_k}.\]
Hypothesis (\ref{eqn:psavthmhyp}) guarantees that for any $k\ge 2$,
\begin{equation}\label{eqn:expbnd1}
\max_{1\le r\le s}b^{(r)}_k\le 2k^\delta.
\end{equation}
Now for each $\ell\in\N$ let $\sigma_\ell\in\Z$ and $\tau_\ell\in [0,1)$ be selected so that $\sigma_\ell\ge 0$ and
\[\sum_{r=1}^s b^{(r)}_\ell\frac{\log p_r}{\log a}=\sigma_\ell+\tau_\ell.\]
Note that this is the same as writing $n_\ell$ in the form $a^{\sigma_\ell+\tau_\ell}$, and doing this makes it technically easier to compare the ratios of these numbers. Let $M$ be the smallest integer greater than $2\log a$. Then given $k\ge 2$, one of the intervals $[m/M,(m+1)/M), 0\le m<M,$ contains at least $k/M$ of the numbers $\{\tau_\ell\}_{1\le\ell\le k}$. Label the numbers which fall in this interval as $\tau_{\ell_1}<\cdots <\tau_{\ell_{k'}}$.

Next set $\sigma'=\max_{1\le i\le k'} \sigma_{\ell_i}$ and for each $1\le i\le k'$ let
\[t_i=a^{\sigma'-\sigma_{\ell_i}}n_{\ell_i}\in\Sigma_a\D.\]
Then for any $1\le i<j\le k'$ we have that
\begin{align}
\log\left(\frac{t_j}{t_i}\right)&=\sum_{r=1}^s\left(b^{(r)}_{\ell_j}-b^{(r)}_{\ell_i}\right)\log p_r+\left(\sigma_{\ell_i}-\sigma_{\ell_j}\right)\log a\label{eqn:logt_i/t_j1}\\
&=\log a\left(\tau_{\ell_j}-\tau_{\ell_i}\right),\nonumber
\end{align}
and this shows that
\begin{equation}\label{eqn:logt_i/t_j2}
0<\log\left(\frac{t_j}{t_i}\right)<\frac{\log a}{M}.
\end{equation}
Next using (\ref{eqn:expbnd1}) we have that
\[|\sigma_{\ell_j}-\sigma_{\ell_i}|\le \frac{s\log p_s}{\log a}\cdot\max_{1\le r\le s}\left(1+\left|b^{(r)}_{\ell_j}-b^{(r)}_{\ell_i}\right|\right)\le \left(\frac{4s\log p_s}{\log a}\right)k^\delta,\]
and so by applying Theorem \ref{thm:BW} to (\ref{eqn:logt_i/t_j1}) we deduce that there are constants $C,\kappa>0$, which depend only on $p_1,\ldots ,p_s$, and $a$, such that
\begin{equation}\label{eqn:logt_i/t_j3}
\log\left(\frac{t_j}{t_i}\right)\ge\frac{C}{k^{\delta\kappa}}.
\end{equation}
To avoid technicalities from here on we will assume that $k\ge \max\{2M,2C^{1/(\delta\kappa)}\}$. Combining (\ref{eqn:logt_i/t_j2}) and (\ref{eqn:logt_i/t_j3}) with the inequalities
\[1+x\le e^x\le 1+2x,~0\le x\le 1,\]
we have that
\begin{equation}\label{eqn:t_j/t_i}
1+\frac{C}{k^{\delta\kappa}}~\le~\frac{t_j}{t_i}<~ 1+\frac{2\log a}{M},~\text{ for all }~1\le i<j\le k'.
\end{equation}

Next we claim that we can always find a number $\gamma\in [0,1)$ satisfying
\[\gamma\in \left[\frac{1}{t_1a^2},\frac{1}{t_1a}\right)\cap\left(\Sigma_a-\Sigma_a\right)\alpha.\]
To see why this is true notice that since $\alpha\not\in\Q$ the point $0$ is an accumulation point of $(\Sigma_a-\Sigma_a)\alpha=\Sigma_a\alpha-\Sigma_a\alpha$. Also this set is symmetric about $0$, so it contains infinitely many points which lie in the interval $(0,1/t_1a^2)$. If $\beta$ is one of these points then we can find an integer $b\in\N$ with $a^b\beta\in [1/t_1a^2,1/t_1a)$, and our claim is verified by taking $\gamma=a^b\beta$.

With $\gamma$ as above, for any $n\in\N$ and for any $1\le i\le k'$ we have from (\ref{eqn:t_j/t_i}) that
\[\frac{1}{a^2}\le t_i\gamma \le t_1\gamma\left(1+\frac{2\log a}{M}\right)<\frac{2}{a}\le 1.\]
Furthermore if $i<k'$ then from the lower bound in (\ref{eqn:t_j/t_i}) we obtain
\[t_{i+1}\gamma-t_i\gamma\ge \frac{t_i\gamma C}{k^{\delta\kappa}}\ge\frac{C}{a^2k^{\delta\kappa}}.\]
Thus for each $n\in\N$ we have that
\[s\left(A(\alpha)-A(\alpha),\frac{C}{a^2k^{\delta\kappa}}\right)\ge k' \ge \frac{k}{M}.\]
Then by Lemma \ref{lem:epssepdiffset} we have
\begin{equation}\label{eqn:epssepa^nsigma}
s\left(A(\alpha),\frac{C}{a^2k^{\delta\kappa}}\right)\ge \left(\frac{k}{2M}\right)^{1/2},
\end{equation}
and this gives that
\[\udim (A(\alpha))\ge\limsup_{k\rar\infty}\frac{\log \left(\left(\frac{k}{2M}\right)^{1/2}\right) }{\log\left(\frac{2a^2k^{\delta\kappa}}{C}\right)}=\frac{1}{2\delta\kappa}>0.\]
Finally Lemma \ref{lem:dimentropy} ensures that there is an ergodic times-$a$ invariant measure $\mu$, supported on $A(\alpha),$ which has positive entropy. By Theorem \ref{thm:EFS} we have that $\mu-$almost every point is well approximable, but since $\mu(\R/\Z\setminus A(\alpha))=0$ this implies that $A(\alpha)$ contains a well-approximable point. This finishes the proof of the theorem.


\section{Preliminaries for the proof of Theorem \ref{thm:2}}
\label{sec:Preliminaries2}

Let $\Psi:\N\rar\R$ be any non-negative function and for each $n\in\N$ define $A_n=A_n(\Psi)\subseteq\R/\Z$ by
\[A_n(\Psi)=\{\alpha\in\R/\Z : \|n\alpha\|\le\Psi(n)\}.\]
Then define $A(\Psi)\subseteq\R/\Z$ by
\[A(\Psi)=\limsup_{n\rar\infty}A_n(\Psi)=\{\alpha\in\R/\Z:\alpha\in A_n\text{ for infinitely many }n\}.\]
In our problem we are interested in the case when $\Psi(n)=|n|_\D^{-1}\psi(n)$, for a pseudo-absolute value $|\cdot|_\D$ and a non-negative monotonic function $\psi:\N\rar\R$. If $\lambda$ denotes Lebesgue measure on $\R/\Z$ then we would like to show for this choice of $\Psi$ that $\lambda (A(\Psi))=1$ depending on whether or not the sum (\ref{eqn:thm2divcond}) diverges. First of all we demonstrate that the divergence or convergence of the sum in question is equivalent to the divergence or convergence of the measures of the corresponding sets $A_n$.
Here and in what follows we write $d_k=n_k/n_{k-1}$ for each $n_k\in\D,k\ge 1$.
\begin{lemma}
	\label{lem:asymp}
If $\D$ is any pseudo-absolute value sequence then for $N\in\N$ we have that
	\begin{equation*}
		\sum_{n=1}^N \frac{1}{\infabs{n}_\mathcal{D}} \quad \stackrel{(i)}{\asymp} \quad N  \mathcal{M}(N) \quad \stackrel{(ii)}{\asymp} \quad \sum_{n=1}^N  \mathcal{M}(n).
	\end{equation*}
Consequently if $\psi:\mathbb{N} \rightarrow \mathbb{R}$ is any non-negative decreasing function then
	\begin{equation}\label{eqn:measuredivcond}
		\sum_{n=1}^{\infty} \lambda \left(A_n\left(\frac{\psi}{|\cdot|_\D}\right)\right) \, = \, \infty \quad \Longleftrightarrow \quad \sum_{n=1}^{\infty} \mathcal{M}(n) \psi(n) \, = \, \infty.
	\end{equation}
\end{lemma}
\begin{proof}For the proof of $(i)$ we have that
\begin{align}
\sum_{n=1}^N\frac{1}{|n|_\D}&=\sum_{k=0}^{\M(N)}n_k\sum_{\substack{n=1\\n_k|n,~n_{k+1}\nmid n}}^N1\nonumber\\
&=\sum_{k=0}^{\M(N)}n_k\sum_{\substack{n\le N/n_k\\d_{k+1}\nmid n}}1\nonumber\\
&=\sum_{k=0}^{\M(N)}n_k\left(\left(1-\frac{1}{d_{k+1}}\right)\frac{N}{n_k}+O(1)\right)\nonumber\\
&=N\sum_{k=0}^{\M(N)}\left(1-\frac{1}{d_{k+1}}\right)+O\left(\sum_{k=0}^{\M(N)}n_k\right).\label{eqn:divlem1}
\end{align}
Now notice that $1/2\le (1-1/d_{k+1})< 1$ for all $k$ and that
\begin{equation}\label{eqn:asymplem1}
\sum_{k=0}^{\M(N)}n_k\le \sum_{k=0}^{\M(N)}\frac{n_{\M(N)}}{2^{\M(N)-k}}\le 2n_{\M(N)}\le 2N.
\end{equation}
As claimed this shows that (\ref{eqn:divlem1}) is bounded above and below by universal constants times $N\M(N)$.

For	$(ii)$ we have that
\begin{align*}
\sum_{n=1}^{N}\M(n)&=\sum_{k=0}^{\M(N)-1}k(n_{k+1}-n_k)+\M(N)(N-n_{\M(N)}+1)\\
&=(N+1)\M(N)-\sum_{k=0}^{\M(N)}n_k
\end{align*}
The latter quantity is clearly less than $2N\M(N)$, and by (\ref{eqn:asymplem1}) it is also greater than a constant times $N\M(N)$.

Finally for the proof of (\ref{eqn:measuredivcond}), first of all suppose that $\psi(m_i)/|m_i|_\D\ge 1/2$ for some infinite increasing sequence of integers $\{m_i\}_{i\in\N}$. Then for each $i$ we have that $A_{m_i}=\R/\Z$ so that the left hand side of (\ref{eqn:measuredivcond}) surely diverges. On the other hand using (ii) we have that
\[\sum_{n=1}^{m_i}\M(n)\psi(n)\ge \psi(m_i)\sum_{n=1}^{m_i}\M(n)\gg |m_i|_\D \left( m_i\M(m_i)\right)\ge \M(m_i),\]
and this tends to infinity with $i$.

Now for the other case assume that there is an $n_0\in\N$ such that $\psi(n)/|n|_\D<1/2$ for all $n\ge n_0$. In this case we have that
\begin{align}\label{eqn:measequiv1}
\lambda \left(A_n\left(\frac{\psi}{|\cdot|_\D}\right)\right)&=\frac{2\psi(n)}{|n|_\D} ~\text{ for all }n\ge n_0.
\end{align}
Now by the monotonicity of $\psi$ together with (i) and (ii) we obtain
\begin{align*}
\sum_{n=n_0}^N\frac{\psi(n)}{|n|_\D}&=\sum_{n=n_0}^N(\psi(n)-\psi(n+1))\sum_{m=n_0}^n\frac{1}{|m|_\D}+\psi (N+1)\sum_{m=n_0}^N\frac{1}{|m|_\D}\\
&\asymp \sum_{n=n_0}^N(\psi(n)-\psi(n+1))\sum_{m=n_0}^n\M(m)+\psi (N+1)\sum_{m=n_0}^N\M(m)\\
&=\sum_{n=n_0}^N\M(n)\psi (n),
\end{align*}
and this together with (\ref{eqn:measequiv1}) finishes the proof of the lemma.
\end{proof}
For any $\Psi$ as above if
\[\sum_{n\in\N}\lambda(A_n(\Psi))<\infty\]
then by the Borel-Cantelli Lemma we have that $\lambda (A(\Psi))=0$. One half of Theorem \ref{thm:2} follows from this observation together with (\ref{eqn:measuredivcond}). Unfortunately the converse of the Borel-Cantelli Lemma is not true in general for the sets $A_n(\Psi)$. In other words there are examples of functions $\Psi$ for which
\[\sum_{n\in\N}\lambda(A_n(\Psi))=\infty\]
and yet $\lambda (A(\Psi))=0$. Duffin and Schaeffer observed this in \cite{DS} and they also showed in the same paper that under certain conditions this type of anomalous behavior can be avoided.
\begin{theorem}{\cite{DS}}\label{thm:DS}
If $\Psi:\N\rar\R$ is a nonnegative function which satisfies
\begin{equation}\label{eqn:thmDShyp1}
\sum_{n\in\N}\Psi(n)=\infty
\end{equation}
and
\begin{equation}\label{eqn:thmDShyp2}
\limsup_{N\rar\infty}\left(\sum_{n=1}^N\frac{\varphi(n)\Psi (n)}{n}\right)\left(\sum_{n=1}^N\Psi(n)\right)^{-1}>0
\end{equation}
then $\lambda (A(\Psi))=1$.
\end{theorem}


\section{Proof of Theorem \ref{thm:2}}
\label{sec:ProofOfTheorem2}

If (\ref{eqn:thm2divcond}) converges then as previously remarked the result of Theorem \ref{thm:2} follows from the Borel-Cantelli Lemma. Therefore we assume that (\ref{eqn:thm2divcond}) diverges. We set $\Psi(n)=\psi(n)/|n|_\D$ and we assume without loss of generality that $\Psi(n)<1/2$ for all but finitely many $n$ (otherwise the conclusion of the theorem is trivial). Then by (\ref{eqn:measuredivcond}) and (\ref{eqn:measequiv1}) we know that (\ref{eqn:thmDShyp1}) is satisfied, so in order to prove Theorem \ref{thm:2} it is sufficient to show that (\ref{eqn:thmDShyp2}) also holds.

First of all we show that there is a universal constant $C>0$ such that
\begin{equation}\label{eqn:philem1}
\sum_{\substack{n=1\\d\nmid n}}^N\frac{\varphi (n)}{n}\ge CN\quad\text{for any }~ d,N\in\N  ~\text{ with }~d\ge 2.
\end{equation}
To verify this we have that
\begin{align}
\sum_{\substack{n=1\\d\nmid n}}^N\frac{\varphi (n)}{n}&=\sum_{n=1}^N\frac{\varphi (n)}{n}-\sum_{\substack{n=1\\d|n}}^N\frac{\varphi (n)}{n}\nonumber\\
&=\sum_{d=1}^N\frac{\mu (d)}{d}\sum_{1\le n\le N/d}1-\sum_{\substack{n=1\\d|n}}^N\frac{\varphi (n)}{n},\nonumber
\end{align}
where $\mu:\N\rar\{\pm 1,0\}$ is the M\"{o}bius function. For the first sum in this expression we use the fact that
\begin{align*}
\sum_{d=1}^N\frac{\mu (d)}{d}\sum_{1\le n\le N/d}1&=N\sum_{d=1}^N\frac{\mu (d)}{d^2}-\sum_{d=1}^N\left\{\frac{N}{d}\right\}\frac{\mu (d)}{d}\\
&=\frac{6N}{\pi^2}-N\sum_{d=N+1}^\infty\frac{\mu (d)}{d^2}-\sum_{d=1}^N\left\{\frac{N}{d}\right\}\frac{\mu (d)}{d}\\
&\ge \frac{6N}{\pi^2}-C_1\log (N+1),
\end{align*}
for some universal constant $C_1>0$. For the second sum we simply use the inequality
\[\sum_{\substack{n=1\\d|n}}^N\frac{\varphi (n)}{n}\le\frac{N}{d}.\]
Together these estimates give
\[\sum_{\substack{n=1\\d\nmid n}}^N\frac{\varphi (n)}{n}\ge \left(\frac{6}{\pi^2}-\frac{1}{d}\right)N-C_1\log (N+1).\]
Now since $d\ge 2$ we have $6/\pi^2-1/d>0$ and therefore there exists an $N_0\in\N$ such that
\[\left(\frac{6}{\pi^2}-\frac{1}{d}\right)N\ge 2C_1\log (N+1)\quad\text{for all }N\ge N_0,\]
which means that
\[\sum_{\substack{n=1\\d\nmid n}}^N\frac{\varphi (n)}{n}\ge\frac{1}{2}\left(\frac{6}{\pi^2}-\frac{1}{2}\right)N\quad\text{for all }N\ge N_0.\]
To take care of the smaller values of $N$ we choose $C_2>0$ so that
\begin{equation}\label{eqn:philem2}
\sum_{\substack{n=1\\d\nmid n}}^N\frac{\varphi (n)}{n}\ge C_2N\quad\text{ for all }N<N_0, d\le N_0.
\end{equation}
This is clearly possible since the summand is always positive and the range of values for both $N$ and $d$ is finite. However if (\ref{eqn:philem2}) holds for all $d\le N_0$ then it also holds for all $d>N_0,$ since the left hand side only depends on $N<N_0$ in those cases. This establishes (\ref{eqn:philem1}) with
\[C=\min\left\{C_2,\frac{1}{2}\left(\frac{6}{\pi^2}-\frac{1}{2}\right)\right\}.\]

For the final part of the proof we have that
\begin{equation}\label{eqn:thm2partialsum1}
	\sum_{n=1}^N \frac{\varphi(n)\psi(n)}{n\infabs{n}_{\mathcal{D}}} \,  = \,  \sum_{n=1}^N
	\left(\psi(n)-\psi(n+1) \right) \sum_{m=1}^n \frac{\varphi(m)}{m\infabs{m}_{\mathcal{D}}}
	+ \psi(N+1)\sum_{m=1}^N \frac{\varphi(m)}{m\infabs{m}_{\mathcal{D}}}.
\end{equation}
We estimate the inner sums here by
\begin{align*}
\sum_{m=1}^n \frac{\varphi(m)}{m\infabs{m}_{\mathcal{D}}}&=\sum_{k=1}^{\M(n)}\sum_{\substack{m=1\\n_k|m,~n_{k+1}\nmid m}}^n\frac{\varphi(m)}{m\infabs{m}_{\mathcal{D}}}\\
&=\sum_{k=1}^{\M(n)}\sum_{\substack{1\le m\le n/n_k\\d_{k+1}\nmid m}}\frac{\varphi(n_km)}{m}\\
&\ge \sum_{k=1}^{\M(n)}\varphi (n_k)\sum_{\substack{1\le m\le n/n_k\\d_{k+1}\nmid m}}\frac{\varphi(m)}{m}\\
&\ge \frac{Cn}{2}\sum_{k=1}^{\M(n)}\frac{\varphi (n_k)}{n_k}.
\end{align*}
By hypothesis (\ref{eqn:thm2.Dhyp}) the last sum here is $\gg \M(n)$ and so by inequality (i) in Lemma \ref{lem:asymp} we have that
\[\sum_{m=1}^n \frac{\varphi(m)}{m\infabs{m}_{\mathcal{D}}}\gg\sum_{m=1}^n\frac{1}{|m|_\D}.\]
This together with (\ref{eqn:thm2partialsum1}) and the monotonicity of $\psi$ gives
\begin{align*}
	\sum_{n=1}^N \frac{\varphi(n)\psi(n)}{n\infabs{n}_{\mathcal{D}}} \,  &\gg \,  \sum_{n=1}^N
	\left(\psi(n)-\psi(n+1) \right) \sum_{m=1}^n \frac{1}{\infabs{m}_{\mathcal{D}}}
	+ \psi(N+1)\sum_{m=1}^N \frac{1}{\infabs{m}_{\mathcal{D}}}\\
&=\sum_{n=1}^N \frac{\psi(n)}{\infabs{n}_{\mathcal{D}}}.
\end{align*}
This shows that (\ref{eqn:thmDShyp2}) is satisfied and we conclude our proof by applying Theorem \ref{thm:DS}.


\section{Concluding remarks}\label{sec:conclusion}
We mentioned in the introduction that hypothesis (\ref{eqn:thm2.Dhyp}) in Theorem \ref{thm:2} is not particularly restrictive. However there are sequences $\D$ for which it fails to hold. To see how one might construct such a sequence, for each $k\ge 0$ let $A_k=2^{k^2}$ and set
\[ 	n_k= \prod_{\, p \, \le A_k} p, \]
where the product is over prime numbers. Then by one of Mertens' Theorems \cite[\S22.7]{HW} we have that
\begin{align*}
\frac{\varphi (n_k)}{n_k}=\prod_{p\le A_k}\left(1-\frac{1}{p}\right)\ll \frac{1}{\log A_k},
\end{align*}
and this implies that
\[\sum_{k=1}^\infty\frac{\varphi (n_k)}{n_k}<\infty.\]
It is clear in this example that if $\D=\{n_k\}$ then (\ref{eqn:thm2.Dhyp}) is not satisfied.

It would be interesting to determine whether or not hypothesis (\ref{eqn:thm2.Dhyp}) can be removed from the proof of Theorem \ref{thm:2}. Indeed another interesting question is to determine whether hypothesis (\ref{eqn:psavthmhyp}) can be removed from the proof of Theorem \ref{thm:psav}. Both of these problems seem to require more than trivial improvements over the techniques which we have presented.

Finally we remark that the ideas in our proof of Theorem \ref{thm:2} can be extended to prove metric results about approximations involving more than one pseudo-absolute value. In particular given two pseudo-absolute value sequences $\D_1$ and $\D_2$ and a monotonic function $\psi:\N\rar\R$ we could give conditions on $\D_1,\D_2,$ and $\psi$ which would guarantee that the inequality
\begin{equation}\label{eqn:metr2seqs}
|n|_{\D_1}|n|_{\D_2}\|n\alpha\|\le\psi (n)
\end{equation}
has infinitely many solutions $n\in\N$ for almost every $\alpha\in\R.$ However the conditions would depend very much on how the sequences $\D_1$ and $\D_2$ intersect. For example if $\D_1=\{2^k\}$ and $\D_2=\{3^k\}$ then by \cite[Theorem 1]{BHV}, inequality (\ref{eqn:metr2seqs}) has infinitely many solutions for almost every $\alpha$ if and only if
\[\sum_{n\in\N}(\log n)^2\psi (n)=\infty.\]
However if $\D_1=\D_2=\{2^k\}$ then by \cite[Theorem 2]{BHV}, the inequality has infinitely many solutions for almost every $\alpha$ if and only if
\[\sum_{n\in\N} n\psi (n)=\infty.\]
This shows that there are two extremes of the problem, and most sequences behave in a way that falls between these two extremes. It doesn't seem readily obvious how to find a nice, tractable divergence condition which will apply in the most general case of metric problems involving more than one pseudo-absolute value. In the case of two pseudo-absolute values this might not be too difficult, but for more than two the problem seems to get complicated quickly.

\providecommand{\bysame}{\leavevmode\hbox to3em{\hrulefill}}
\providecommand{\MR}{\relax\ifhmode\unskip\space\fi MR }
\providecommand{\MRhref}[2]{%
  \href{http://www.ams.org/mathscinet-getitem?mr=#1}{#2}
} \providecommand{\href}[2]{#2}

\end{document}